\documentclass{article}
\usepackage[utf8]{inputenc}
\usepackage[margin=1in]{geometry}
\usepackage{xcolor}
\usepackage{amsmath} 
\usepackage{amsfonts}
\usepackage{graphicx}
\usepackage{multirow}
\usepackage{amsthm}
\usepackage{natbib}

\newcommand{\R}{\mathbb R} %Reals
\newcommand{\QQ}{\mathcal{Q}} 

\newcommand{\I}{\mathcal{I}}

\newcommand{\bx}{\boldsymbol{x}}
\newcommand{\bh}{\boldsymbol{h}}
\newcommand{\by}{\boldsymbol{y}}
\newcommand{\bv}{\boldsymbol{v}}
\newcommand{\bmm}{\boldsymbol{m}}

\newcommand{\bpi}{\boldsymbol{\pi}}

\newtheorem{lemma}{Lemma}
\newtheorem{definition}{Definition}
\newtheorem{theorem}{Theorem}
\newtheorem{case}{Case}

\title{On the Identifiability of Latent Class Models for Multiple-Systems Estimation}
\author{Serge Aleshin-Guendel\\
Department of Biostatistics, University of Washington\\
aleshing@uw.edu}
\begin{document}

\maketitle
\begin{abstract}
Latent class models have recently become popular for multiple-systems estimation in human rights applications. However, it is currently unknown when a given family of latent class models is identifiable in this context. We provide necessary and sufficient conditions on the number of latent classes needed for a family of latent class models to be identifiable. Along the way we provide a mechanism for verifying identifiability in a class of multiple-systems estimation models that allow for individual heterogeneity. 

\end{abstract}

\section{Introduction}
Multiple-systems estimation, also known as capture-recapture in ecological settings, is an approach to estimating hard to reach population sizes which has been used in a number of fields, including epidemiology, official statistics, and human rights \citep{Hook_1995a, Bird_2018, Ball_2019b}. In this setting, multiple sources have incompletely sampled from a closed population of interest, and individuals sampled by more than one source are able to be uniquely identified. The population size is estimated based on the observed overlap of the sources using a model that describes how individuals were sampled by the sources. A common problem in multiple-systems estimation is that individuals may be sampled heterogeneously, i.e. different individuals may have different probabilities of being sampled by each source. In order to come up with reliable estimates of the population size in these settings, one needs to take this heterogeneity into account in the model describing the sampling process of the sources. 

Two classic models that incorporate individual heterogeneity are the $M_{th}$ and $M_{h}$ models \citep{Otis_1978}. The $M_{th}$ model assumes that each individual is independently sampled by each source, conditional on latent probabilities of being sampled by each source. The $M_h$ model additionally assumes that an individual has the same probability of being sampled by each source. The $M_h$ model can be plausible in ecological settings where researchers have the ability to design experiments where animals have the same probability of being sampled by each source. However, the assumption that an individual has the same probability of being sampled by each source is typically not plausible in human populations where sources often use convenience samples \citep{Ball_2019b}. This motivates the use of $M_{th}$ models in such settings. In particular, \cite{Manrique-Vallier_2016} recently proposed a family of latent class models with a large number of classes, a type of $M_{th}$ model, which has become popular in human rights and other human population settings \citep{Sadinle_2018, Ball_2018b, Ball_2018c, Manrique-Vallier_2019b, Angel_2019, Ball_2019a, Doshi_2019, Okiria_2019}. 

When using either $M_{h}$ or $M_{th}$ models, it is known that one must restrict oneself to a parametric family of models for identification, which has generated a literature characterizing identifiability in $M_{h}$ models \citep{Huggins_2001, Link_2003, Holzmann_2006, Link_2006}. Once identifiability is settled for a family of models, one can begin discussing properties of population size estimates using that family, such as consistency \citep{Sanathanan_1972} or finite-sample risk \citep{Johndrow_2019}. Currently, a literature does not exist characterizing identifiability in $M_{th}$ models, even though a large number of parametric $M_{th}$ model families have been proposed in the literature \citep{Agresti_1994, Coull_1999, Fienberg_1999, Pledger_2000, Bartolucci_2004, Durban_2005, King_2008}.

In this paper we partially close the gap between theory and methodology for $M_{th}$ models through two contributions. The first contribution is a mechanism for verifying identifiability in $M_{th}$ models based on moments of the distribution for the latent sampling probabilities. The second contribution is a necessary and sufficient condition for families of latent class models to be identified. Our result shows that recent applications using latent class models for multiple-systems estimation have been based on nonidentifiable families of models. 

\section{Preliminaries}
\subsection{Model Description}
Suppose $K\geq2$ sources sample individuals in a closed population of size $N$, of which only $n<N$ are sampled by at least one source. We let $H=\{0,1\}^K$ denote the possible inclusion patterns of individuals in the sources, $H^*=H\setminus \{0\}^K$ denote the possible inclusion patterns of the $n$ individuals sampled in at least one source, and $\bx_i\in H$ denote the inclusion pattern for individual $i$. For example, if $K=4$ and $\bx_i=(1, 1, 1,0)$, individual $i$ was sampled by sources $1, 2,$ and $3$, but not by source $4$. The $\bx_i$ can be aggregated into a $2^K$ contingency table, with cells indexed by $\bh\in H$ and cell counts $n_{\bh}=\sum_{i=1}^NI(\bx_i=\bh)$. We do not observe the count for the cell $\{0\}^K$, $n_{(0,\ldots,0)}=N-n$. The target of inference is the population size $N$.

In this article, we assume that that the individuals' inclusion patterns follow an $M_{th}$ model \citep{Otis_1978}, i.e. 
\begin{align}
\label{eq:mth}
\begin{split}
    x_{i,k}\mid \lambda_{i,k} &\stackrel{iid}{\sim} \textsc{Bernoulli}(\lambda_{i,k}),\\
    (\lambda_{i,1},\ldots, \lambda_{i,K})&\stackrel{iid}{\sim} Q,
\end{split}
\end{align}
where $Q\in\QQ$ and $\QQ$ is a family of mixing distributions on $(0,1]^K$ for the latent sampling probabilities. Under this model, conditional on an individual's sampling probabilities, $(\lambda_{i,1},\ldots, \lambda_{i,K})$, the individual is sampled by each source $k$ with probability $\lambda_{i,k}$, independently of all other sources. By imposing the restriction that $\lambda_{i,k}\in(0,1]^K$ for each individual $i$ and source $k$, we are assuming that each of the $N$ individuals has non-zero probability of being sampled by at least one source. The $M_{h}$ model \citep{Otis_1978} is a submodel of the $M_{th}$ model which further assumes that the latent sampling probabilities are equal for all sources, i.e. $\lambda_{i,1}=\cdots= \lambda_{i,K}$.

\subsection{Identifiability in $M_{th}$ Models}
Marginalizing over $Q$ in \eqref{eq:mth}, we find that the complete $2^K$ contingency table of counts is multinomially distributed, i.e.
\begin{equation}
    \label{eq:completedatadist}
    (n_{\bh})_{\bh\in H}\mid N, \bpi_Q \sim \textsc{Multinomial}(N, \bpi_Q)
\end{equation}
where $\bpi_Q=(\pi_{Q,\bh})_{\bh\in H}$, $\pi_{Q,\bh}=E_Q\{\prod_{k=1}^K \lambda_{k}^{h_k}(1-\lambda_{k})^{1-h_k}\}$, and $E_Q$ denotes an expectation with respect to the mixing distribution $Q$. Throughout this article when a vector or matrix is indexed by $\bh\in H$ or $H^*$ we use the order given by viewing $\bh$ as binary digits. For example, when $K=3$ we order $H$ as $(0,0,0)$, $(0,0,1)$, $(0,1,0)$, $(0,1,1)$, $(1,0,0)$, $(1,0,1)$, $(1,1,0)$, $(1,1,1)$.

The likelihood in \eqref{eq:completedatadist} can be decomposed as 
\begin{align}
\label{eq:decomp}
    n\mid N, \pi_{Q,0} &\sim\textsc{Binomial}(N, 1-\pi_{Q,0}), \nonumber\\
    (n_{\bh})_{\bh\in H^*}\mid n, \tilde{\bpi}_Q &\sim \textsc{Multinomial}(n, \tilde{\bpi}_Q),
\end{align}
where $\pi_{Q,0}=\pi_{Q,(0,\ldots,0)}$, $\tilde{\bpi}_Q=(\tilde{\pi}_{Q,\bh})_{\bh\in H^*}$, and $\tilde{\pi}_{Q,\bh}=\pi_{Q,\bh}/(1-\pi_{Q,0})$. The multinomial likelihood for the observed cell counts, $(n_{\bh})_{\bh\in H^*}$, conditional on their sum, $n$, in \eqref{eq:decomp} is referred to as the conditional likelihood \citep{Fienberg_1972}. Intuitively, in order to estimate $N$, the conditional cell probabilities, $\tilde{\bpi}_Q$, estimated from the conditional likelihood need to determine the missing cell probability, $\pi_{Q,0}$. The following definition of identifiability, modified from \cite{Link_2003}, codifies this intuition. If a family of distributions $\QQ$ is identifiable according to this definition, then $N$ can be consistently estimated within $\QQ$ \citep{Sanathanan_1972}.
\begin{definition}
\label{def:ident}
A family of distributions $\QQ$ on $(0,1]^K$ is identifiable if, for $Q,R\in\QQ$, $\tilde{\bpi}_Q=\tilde{\bpi}_R$ implies that $\pi_{Q,0}=\pi_{R,0}$.
\end{definition}

\subsection{Motivating Example: Latent Class Models}
\label{sec:lcmprelim}
Latent class models are a classical tool for the analysis of multivariate categorical data that describe populations which can be stratified into $J$ classes, in which the latent sampling probabilities are homogeneous for individuals within each class \citep{Goodman_1974, Haberman_1979}. They form a special case of the $M_{th}$ model, where the mixing distribution is a discrete finite mixture, and have been used for multiple-systems estimation many times \citep{Agresti_1994, Coull_1999, Pledger_2000, Bartolucci_2004, Manrique-Vallier_2016}. We denote the family of latent class models with $J$ classes by $\QQ_J = \{ Q = \sum_{j=1}^J\nu_{Q, j}\prod_{k=1}^K\delta_{\lambda_{Q,jk}} \mid \nu_{Q, j}\geq 0, \sum_{j=1}^J\nu_{Q, j}=1, \lambda_{Q,jk}\in(0,1]^K\}$. It is currently unknown when $\QQ_J$ is identified.
 
\section{Verifying Identifiability via Mixed Moments}
In this section, we aim to provide a mechanism for directly checking Definition \ref{def:ident}, to verify identifiability of a given family $\QQ$. Before proving the main theorem of this section, we have the following lemma, which tells us that cell probabilities for any $M_{th}$ model only depends on the mixing distribution, $Q$, through mixed moments of $Q$. 
\begin{lemma}
\label{lemma:mixed}
For any $\bh\in H^*$, $\pi_{Q,\bh}=\sum_{\bh'\in H^*} c_{\bh, \bh'}m_{Q,\bh'}$ where $c_{\bh, \bh'}=(-1)^{\sum_{k=1}^K h'_k - h_k}\prod_{k=1}^K I(h_k \leq h'_k)$ and $m_{Q,\bh'}=E_Q(\prod_{k=1}^K \lambda_k^{h_k'})$.
\end{lemma}
\begin{proof}
    For all $\bh\in H^*$, $\prod_{k=1}^K \lambda_k^{h_k}(1-\lambda_k)^{1-h_k}=
    \sum_{\bh'\in H^*} c_{\bh, \bh'} \prod_{k=1}^K \lambda_k^{h'_k}$ by an application of the multi-binomial theorem. The result follows from taking the expectation over both sides with respect to $Q$.
\end{proof}

We can restate Lemma \ref{lemma:mixed} in matrix form. Letting  $\bpi_Q^*=(\pi_{Q,\bh})_{\bh\in H^*}$ and $\bmm_Q=(m_{Q,\bh})_{\bh\in H^*}$, we have that $\bpi_Q^*=C\bmm_Q$, where $C=(c_{\bh, \bh'})_{\bh\in H^*, \bh'\in H^*}$. $C$ is invertible as it is upper triangular with non-zero diagonal entries. We are now ready to prove Theorem \ref{theorem:moments}.
\begin{theorem}
\label{theorem:moments}
For any two distributions $Q,R$ on $(0,1]^K$, $\tilde{\bpi}_Q=\tilde{\bpi}_R$ is equivalent to $\bmm_Q=A \bmm_R$ for some $A>0$. 
\end{theorem}
\begin{proof}
$\tilde{\bpi}_Q=\tilde{\bpi}_R$ is equivalent to $\bpi_Q^*/(1-\pi_{Q,0})=\bpi_R^*/(1-\pi_{R,0}).$ Rearranging terms we have that $\bpi_Q^*=\bpi_R^*(1-\pi_{Q,0})/(1-\pi_{R,0}),$ and thus $\bpi_Q^*=A\bpi_R^*$, where $A=(1-\pi_{Q,0})/(1-\pi_{R,0})>0$. Using Lemma \ref{lemma:mixed}, this is equivalent to $C\bmm_Q=AC\bmm_R$, and thus $\bmm_Q=A \bmm_R$ due to the invertibility of $C$. 
\end{proof}

The immediate consequence of Theorem \ref{theorem:moments} is that to verify identifiability of a family $\QQ$, one can demonstrate that if $\bmm_Q=A \bmm_R$ for some $Q,R\in\QQ$, then $\pi_{Q,0}=\pi_{R,0}$. We use this mechanism in the next section to characterize when latent class models are identifiable.

\section{When are Latent Class Models Identifiable?}
\subsection{The Answer}
To provide necessary and sufficient conditions for the family of $J$-class latent class models, $\QQ_J$, to be identifiable, we restrict the family defined in Section \ref{sec:lcmprelim} to 
$\QQ_J = \{ Q = \sum_{j=1}^J\nu_{Q, j}\prod_{k=1}^K\delta_{\lambda_{Q,jk}} \mid \nu_{Q, j}\geq 0, \sum_{j=1}^J\nu_{Q, j}=1, \lambda_{Q,jk}\in(0,1]^K , \lambda_{Q,jk}\neq \lambda_{Q,j'k} \text{ for } j\neq j'\}$.
This restriction makes the mild assumption that each class' sampling probabilities are distinct, which simplifies the proof of Theorem \ref{thm:lcmthm}. Loosening this restriction could only make the conditions on $J$ for $\QQ_J$ to be identifiable stricter, and thus the conclusions we reach in the following section would still stand for families where this restriction is violated.

There are $J(K+1) - 1$ parameters in $\QQ_J$, thus when $\QQ_J$ is  identifiable, $J$ satisfies $J(K+1) - 1\leq 2^K -2$, as the conditional cell probabilities, $\tilde{\bpi}_Q$, are $2^K -2$ dimensional. However, we now prove that $J$ must satisfy a stricter condition for $\QQ_J$ to be identifiable. 

\begin{theorem}
\label{thm:lcmthm}
$\QQ_J$ is identifiable iff $2J\leq K$.
\end{theorem}
\begin{proof}
We will first show that if $2J\leq K$, then $\QQ_J$ is identifiable. The proof of this direction is similar in spirit to the proofs of Theorem 2 in \cite{Holzmann_2006} and Theorem 1 in \cite{Pezzott_2019}, which were both concerned with characterizing the identifiability of the $M_h$ analogue of $\QQ_J$. Assume $2J\leq K$, and let $Q,R\in\QQ_J$ such that $\bmm_Q=A \bmm_R$ for some $A>0$, so that we have the following system of equations:
\begin{equation}
\label{eq:firstsyst}
    \sum_{j=1}^J\nu_{Q,j}\prod_{k=1}^K\lambda_{Q,jk}^{h_k} - A
    \sum_{j=1}^J\nu_{R,j}\prod_{k=1}^K\lambda_{R,jk}^{h_k}=0 \quad (\bh\in H^*).
\end{equation}
Let $\I_Q=\{j\mid \lambda_{Q,j}\not\in(\lambda_{R,1},\ldots, \lambda_{R,J})\}$ and $\I_R=\{j\mid \lambda_{R,j}\not\in(\lambda_{Q,1},\ldots, \lambda_{Q,J})\}$, where $\lambda_{Q,j}=(\lambda_{Q,j1},\ldots, \lambda_{Q,jK})$ and $\lambda_{R,j}=(\lambda_{R,j1},\ldots, \lambda_{R,jK})$. We can then rewrite \eqref{eq:firstsyst} as 
\begin{equation}
\label{eq:secondsyst}
    \sum_{j=1}^Jy_j\prod_{k=1}^K\lambda_{Q,jk}^{h_k} - A
    \sum_{i\in \I_R}^J\nu_{R,j}\prod_{k=1}^K\lambda_{R,jk}^{h_k}=0 \quad (\bh\in H^*),
\end{equation}
where $y_j=\nu_{Q,j}$ if $j\in\I_Q$ and $y_j=\nu_{Q,j}-A\nu_{R,j'}$ for some $j'\in\{1,\ldots, J\}\setminus\I_R$ otherwise. Letting $m=|\I_R|=|\I_Q|$ and labelling the elements of $\I_R$ as $i_1,\ldots,i_m$, the system of equations in \eqref{eq:secondsyst} can be written in matrix form as $\Lambda \by=0$, where
\begin{equation*}
\Lambda = 
\begin{pmatrix}
\lambda_{Q,1K} &  \cdots & \lambda_{Q,JK}&\lambda_{R,i_1K}&\cdots&\lambda_{R,i_mK} \\
\vdots  & \ddots& \vdots&\vdots  & \ddots& \vdots \\
\prod_{k=1}^K\lambda_{Q,1k}^{h_k}&\cdots& \prod_{k=1}^K\lambda_{Q,Jk}^{h_k}&\prod_{k=1}^K\lambda_{R,i_1k}^{h_k}&\cdots& \prod_{k=1}^K\lambda_{R,i_mk}^{h_k} \\
\vdots& \ddots & \vdots&\vdots  & \ddots& \vdots\\
\prod_{k=1}^K\lambda_{Q,1k} &\cdots & \prod_{k=1}^K\lambda_{Q,Jk}^{h_k}&\prod_{k=1}^K\lambda_{R,i_1k}&\cdots&\prod_{k=1}^K\lambda_{R,i_mk}
\end{pmatrix}, \quad
\by = 
\begin{pmatrix}
y_1  \\
\vdots  \\
y_J\\
-A\nu_{R,i_1}\\
\vdots\\
-A\nu_{R,i_m}
\end{pmatrix},
\end{equation*}
and the rows of $\Lambda$ are indexed by $\bh\in H^*$. In Appendix 1, we prove that $\Lambda$ is full rank, and thus $\by=0$, for any $m \in\{0,\ldots, J\}$. The proof of this direction concludes by examining three possible cases.

\begin{case}
Suppose $m=0$, i.e. for each $j\in\{1,\ldots, J\}$, there exists some $j'\in\{1,\ldots, J\}$ such that $\lambda_{Q,j}=\lambda_{R,j'}$ and $\nu_{Q,j}=A\nu_{R,j'}$. As $\sum_{j=1}^J \nu_{Q,j}=\sum_{j=1}^J \nu_{R,j}=1$, this implies that $A=1$ and thus $\pi_{Q,0}=\pi_{R,0}$.
\end{case}
\begin{case}
Suppose $m\in\{1,\ldots, J-1\}$, i.e. for each $j\in\{1,\ldots, J\}\setminus\I_Q$, there exists some $j'\in \{1,\ldots, J\}\setminus\I_R$ such that $\lambda_{Q,j}=\lambda_{R,j'}$ and $\nu_{Q,j}=A\nu_{R,j'}$. Further, for each $j\in\I_Q$ and $j'\in \I_R$ $\nu_{Q,j}=\nu_{R,j'}=0$. We can thus ignore the classes $j\in\I_Q$ and $j'\in \I_R$. As $\sum_{j=1}^J \nu_{Q,j}=\sum_{j=1}^J \nu_{R,j}=1$, this implies that $A=1$ and thus $\pi_{Q,0}=\pi_{R,0}$.
\end{case}
\begin{case}
Suppose $m=J$, i.e. for each $j\in\{1,\ldots, J\}$, there exists no $j'\in \{1,\ldots, J\}$ such that $\lambda_{Q,j}=\lambda_{R,j'}$. Then $\nu_{Q,j}=\nu_{R,j}=0$ for $j\in\{1,\ldots, J\}$, which is a contradiction.
\end{case}

We will now show that if $2J> K$, then $\QQ_J$ is not identifiable. To do so we will provide explicit $Q,R\in\QQ_J$ such that $\pi_{Q,0}\neq\pi_{R,0}$, but $\bmm_Q=A \bmm_R$ for $A>0$. This counterexample is modified from \cite{Tahmasebi_2018}, who studied identifiability of families of latent class models outside of the multiple-systems estimation context where $n_{(0,\ldots,0)}$ is observed. Choose $J$ such that $2J> K$. For $j\in\{1,\ldots, J\}$, let $\nu_{Q,j}=\binom{2J}{2j}/(2^{2J-1}-1)$ and $\nu_{R,j}=\binom{2J}{ 2j-1}/(2^{2J-1})$. For $j\in\{1,\ldots, J\}$ and $k\in\{1,\ldots, K\}$, let $\lambda_{Q,jk}=\alpha(2j)$ and $\lambda_{R,jk}=\alpha(2j - 1)$ where $0<\alpha <1 / (2J)$. We thus have that $Q,R\in\QQ_J$, where clearly $Q\neq R$. In Appendix 2 we prove that for these choices of $Q,R$, $\bmm_Q=A \bmm_R$ for $A>0$ such that $A\neq 1$, and thus $\pi_{Q,0}\neq\pi_{R,0}$.
\end{proof}

\subsection{The Implications for the Use of Latent Class Models}
\label{sec:implications}
Recently, \cite{Manrique-Vallier_2016} proposed to use a family of latent class models with an infinite number of classes, i.e. $\QQ_{\infty}=\cup_{J=1}^{\infty}\QQ_J$, for multiple-systems estimation. In practice, \cite{Manrique-Vallier_2016} restricted the actual family used to $\QQ_{J^*}$ for some large $J^*$, for computational purposes. Theorem \ref{thm:lcmthm} tells us that such a family is nonidentifiable if $2J^*>K$. \cite{Manrique-Vallier_2016} suggested setting $J^*=K$, which always results in a nonidentifiable family. In the $\texttt{R}$ \citep{R_2019} package $\texttt{LCMCR}$  \citep{LCMCR_2020} which implements the methodology of \cite{Manrique-Vallier_2016}, the default value of $J^*$ is $5$. Unless one is working with at least $K=10$ sources, which is rare outside of ecological applications, the family being used will not be identifiable. Extensions of \cite{Manrique-Vallier_2016}, such as \cite{Manrique-Vallier_2019b} and \cite{Kang_2020}, share the same problem with nonidentifiability when too many latent classes are used.

In their discussion, \cite{Manrique-Vallier_2016} write, ``[a]s \cite{Fienberg_1972} warns, multiple-recapture estimation --- as any other extrapolation technique --- relies on the untestable assumption that the model that describes the observed counts also applies to the unobserved ones." However, the problem is graver than this when working with a nonidentifiable family $\QQ$, as there can be multiple models that describe the observed counts. For example in the simplest case, consider data from $K=2$ sources generated from the two-class latent class model $Q$ with parameters given in Table \ref{tab:expars}. Under $Q$, $\tilde{\pi}_{Q,(0,1)}=0.276$, $\tilde{\pi}_{Q,(1,0)}=0.276$, $\tilde{\pi}_{Q,(1,1)}=0.448$, and $\pi_{Q,0}=0.316$. However, there is another two-class latent class model $R$, with parameters given in Table \ref{tab:expars}, such that $\tilde{\bpi}_{Q}=\tilde{\bpi}_{R}$ but $\pi_{R,0}=0.219$. Because the family $\QQ_2$ is not identified, if we try to perform estimation within $\QQ_2$, which contains the true data generating model, there is no guarantee that we can estimate well, in any traditional sense, the cell probabilities and population size which generated the data. In particular, nonidentifiability precludes consistent estimation as ``there will be uncertainty in parameter estimates that is not washed out as more data are collected" \citep{Linero_2017}. The proof of Theorem \ref{thm:lcmthm} shows us that such an example can be constructed whenever $2J>K$.

\begin{table}[ht]
\centering
\caption{Parameters of two latent class models which produce identical conditional cell probabilities, but different missing cell probabilities}
\label{tab:expars}
\begin{tabular}{lllllll}
    & $\nu_{1}$ & $\nu_{2}$ & $\lambda_{11}$ & $\lambda_{12}$ & $\lambda_{21}$ & $\lambda_{22}$ \\
    
$Q$ & 0.5       & 0.5       & 0.2475         & 0.2475         & 0.7425         & 0.7425   \\
$R$ & 0.8571429     & 0.1428571     & 0.495          & 0.495          & 0.99           & 0.99           
\end{tabular}
\end{table}

For the past several decades, multiple-systems estimation has been used to estimate hard to reach population sizes in sensitive human rights contexts \citep{Bird_2018, Ball_2019b}. This has resulted in the use of population size estimates to influence public policy and inform criminal tribunals in some cases \citep{Ball_2002b,Ball_2018a, Bird_2018, Ball_2019b}. For multiple-systems estimation to be used in such important contexts, the underlying methods must be well understood statistically. Since its publication, the latent class model of \cite{Manrique-Vallier_2016} has been used to estimate the sizes of the following populations: civilians killed in the Salvadoran civil war \citep{Sadinle_2018},
people disappeared on 17--19 May 2009 in Sri Lanka \citep{Ball_2018b},
women held in sexual slavery by the Japanese military during World War II in Palembang, Indonesia \citep{Ball_2018c},
civilians killed in the Peruvian internal conflict between 1980--2000 \citep{Manrique-Vallier_2019b},
social movement leaders killed in Colombia \citep{Angel_2019}, people killed in drug-related violence in the Philippines \citep{Ball_2019a},
people who inject drugs, men who have sex with men, and female sex workers in Kumpala, Uganda \citep{Doshi_2019}, and female sex workers in South Sudan \citep{Okiria_2019}. We could only find information on the number of latent classes used in two of these applications. \cite{Doshi_2019} had $K=3$ sources and used $J^*=10$ latent classes. \cite{Ball_2019a} used $J^*=5$ latent classes to produce results for six different strata, in which four of the strata had less than $K=10$ sources. Thus both of these applications presented results using nonidentifiable families of latent class models. In all of the other applications there were there less than $K=10$ sources. Thus, if the default setting of $J^*=5$ in the $\texttt{R}$ package $\texttt{LCMCR}$ was used, or any other $J^*$ not satisfying Theorem \ref{thm:lcmthm}, none of the families used were identified. Moving forward, we believe that it is imperative that families of models used for multiple-systems estimation in such sensitive contexts are known to be identified. 

\bibliographystyle{biometrika} 
\bibliography{main.bib}

\appendix
\section*{Appendix 1}
\label{sec:lemmaproof}
We will prove that $\Lambda$ is full rank for any $m\in\{0,\ldots, J\}$ by proving a stronger result. Recall that $K\geq 2$ and let $x_{\ell k}\in(0,1)$ for $\ell\in\{1,\ldots, K\}$ and $k\in\{1,\ldots, K\}$, such that $x_{\ell k}\neq x_{\ell k'}$ for $k\neq k'$. Let 
\begin{equation*}
X^K = 
\begin{pmatrix}
x_{1K} &  \cdots & x_{KK} \\
\vdots  & \ddots& \vdots \\
\prod_{k=1}^Kx_{1k}^{h_k}&\cdots& \prod_{k=1}^Kx_{Kk}^{h_k}\\
\vdots& \ddots & \vdots\\
\prod_{k=1}^Kx_{1k} &\cdots & \prod_{k=1}^Kx_{Kk}^{h_k}
\end{pmatrix},
\end{equation*}
where the rows of $X^K$ are indexed by $\bh\in H^*$. We will show that $X^K$ is full rank by induction on $K$. This implies that $\Lambda$ is full rank, as $J+m\leq 2J\leq K$ by assumption for any $m\in\{0,\ldots, J\}$.

For the base case when $K=2$, verifying $X^2$ is full rank is straightforward.
%\begin{equation*}
%X^2 = 
%\begin{pmatrix}
%x_{1,2} &  x_{2,2}  \\
%x_{1,1} &  x_{2,1} \\
%x_{1,1}\ x_{1,2}  & x_{2,1}\ x_{2,2} .
%\end{pmatrix}.
%\end{equation*}
%Suppose $\bv\in\R^{2\times 1}$ such that $X^2\bv=0$. Then $v_1x_{1,1}+v_2x_{2,1}=0$ implies that $v_1x_{1,1}=-v_2x_{2,1}$. Plugging this into $v_1x_{1,1} x_{1,2}  + v_2x_{2,1} x_{2,2}=0$ we find that $v_2x_{2,2}(x_{2,1}-x_{1,1})=0$ and thus $v_2=0$ and $v_1=0$. 
Assume that $X^{K-1}$ is full rank. Let $\bv\in\R^{K\times 1}$ be such that $X^K\bv=0$. For each $\bh\in \{\bh'\in H^*\mid h_K'=0\}$ we have that $v_K\prod_{k=1}^{K-1}x_{Kk}^{h_k}=-\sum_{\ell=1}^{K-1}v_{\ell}\prod_{k=1}^{K-1}x_{\ell k}^{h_k}$, which implies that $\sum_{\ell=1}^{K-1}v_{\ell}(x_{\ell K}-x_{KK})\prod_{k=1}^{K-1}x_{\ell k}^{h_k}=0$. For $\ell\in\{1,\ldots, K-1\}$, let $v_{\ell}'=v_{\ell}(x_{\ell K}-x_{KK})$ and $\bv'=(v_{1}',\ldots, v_{K-1}')$. This leads to the system of equations $X^{K-1}\bv'=0$. By the inductive assumption, $\bv'=0$. Since $x_{\ell K}\neq x_{KK}$ for  $\ell\in\{1,\ldots, K-1\}$, we have that $v_{\ell}=0$ for $\ell\in\{1,\ldots, K-1\}$, and thus $v_K=0$.

\section*{Appendix 2}
\label{sec:lemmaproof2}
We will now prove that $m_{Q,\bh}=A m_{R,\bh}$ for all $\bh\in H^*$, where $A=(2^{2J-1})/(2^{2J-1}-1) \neq 1$. Define the function $h(x)=(1-e^{\alpha x})^{2J}=\sum_{i=0}^{2J}\binom{2J}{i}(-1)^ie^{\alpha i x}$. For $t\in\{1,\ldots, K\}$, we can differentiate the series representation of $h$ to find that $h^{(t)}(x)= \sum_{i=0}^{2J}\binom{2J}{i}(-1)^i(\alpha i)^t e^{\alpha i x}$ and thus $h^{(t)}(x)\vert_{x=0}= \sum_{i=0}^{2J}\binom{2J}{ i}(-1)^i(\alpha i)^t=\sum_{i=1}^{2J}\binom{2J}{ i}(-1)^i(\alpha i)^t.$ We can alternatively differentiate the non-series representation of $h$ using the fact that $t\leq K <2J$ and the chain rule for higher order derivatives to find that $h^{(t)}(x)\vert_{x=0}=0$. Let $\bh\in H^*$ and $t=\sum_{k=1}^K h_k\in\{1,\ldots, K\}$. The desired result follows as
\begin{align*}
m_{Q,\bh}-A m_{R,\bh}&=
    \sum_{j=1}^J\nu_{Q, j} \prod_{k=1}^K\lambda_{Q,jk}^{h_k}-A\sum_{j=1}^J\nu_{R, j} \prod_{k=1}^K\lambda_{R,jk}^{h_k} \\
    &=\sum_{j=1}^J \binom{2J}{2j}(2^{2J-1}-1)^{-1}
    \prod_{k=1}^K\{\alpha(2j)\}^{h_k}- A\sum_{j=1}^J
    \binom{2J}{2j-1}(2^{2J-1})^{-1}
    \prod_{k=1}^K \{\alpha(2j - 1)\}^{h_k}\\
    &=(2^{2J-1}-1)^{-1}\sum_{i=1}^{2J} \binom{2J}{ i} (-1)^i(\alpha i)^t=(2^{2J-1}-1)^{-1}\{h^{(t)}(x)\vert_{x=0}\}=0. 
\end{align*} 

\end{document}